\documentclass[a4paper,12pt]{article}

\usepackage{lineno,hyperref}
\usepackage{amsmath,amsfonts,amssymb,amsthm,amscd,nccmath}
\usepackage{graphicx}
\usepackage[numbers]{natbib}
\usepackage{color}
\usepackage{mathtools}
\usepackage{enumerate}
\usepackage{lipsum}
\usepackage{hyperref}
\usepackage{setspace}
\usepackage[margin=2.5cm]{geometry}

\newtheorem{Theorem}{Theorem}[section]

\newtheorem{Proposition}[Theorem]{Proposition}
\newtheorem{Definition}[Theorem]{Definition}

\newtheorem{Lemma}[Theorem]{Lemma}

\newtheorem{Corollary}[Theorem]{Corollary}

\newtheorem{Fact}[Theorem]{Fact}

\def\forkindep{\mathrel{\raise0.2ex\hbox{\ooalign{\hidewidth$\vert$\hidewidth\cr\raise-0.9ex\hbox{$\smile$}}}}}

\begin{document}

\title{Definable Topological Dynamics of $SL_2(\mathbb{C}((t)))$}

\author{Thomas Kirk \\ University of Central Lancashire}
\maketitle



\begin{abstract}
We initiate a study of definable topological dynamics for groups definable in metastable theories. Specifically, we consider the special linear group $G = SL_2$ with entries from $M = \mathbb{C}((t))$; the field of formal Laurent series with complex coefficients. We prove such a group is not definably amenable, find a suitable group decomposition, and describe the minimal flows of the additive and multiplicative groups of $\mathbb{C}((t))$. The main result is an explicit description of the minimal flow and Ellis Group of $(G(M),S_G(M))$ and we observe that this is not isomorphic to $G/G^{00}$, answering a question as to whether metastability is a suitable weakening of a conjecture of Newelski.
\end{abstract}

\section{Introduction}

Given a model $M$ and a definable group $G$, we can obtain an action of $G$ on $S_G(M)$, the space of complete types concentrating on $G$. This is a definable $G$-flow, and further a $G(M)$-flow in the context of classical topological dynamics. Since $S_G(M)$ is Hausdorff, locally compact, we can construct a unique (up to isomorphism) Ellis Group from the flow $(G(M),S_G(M))$; the details of this construction can be found in \textbf{\cite{TopDynStableGroup}} and \textbf{\cite{Newelski2009}}. For details on the dynamical systems and Ellis Groups, see \textbf{\cite{Ellis}} and \textbf{\cite{Aus}}.

In the stable setting, it was proven in \textbf{\cite{TopDynStableGroup}} that $S_G(M)$ has a unique $G$-invariant minimal subset which is precisely the set of generic types of $G$. Newelski shows a relationship between this set and the quotient $G/G^{00}$, and conjectured that this relationship should extend outside of the stable setting to all $NIP$ theories. Preliminary work in \textbf{\cite{Newelski2012}} extends to the $o$-minimal setting with $G$ definably compact, with \textbf{\cite{2015APNY}}, \textbf{\cite{CONVERSANO2012}} and \textbf{\cite{CS_NIP}} extending to definably amenable groups in $NIP$ theories.

A counterexample to the conjecture was found in \textbf{\cite{SL2R}} for the case $G(M) = SL_2(\mathbb{R})$, and a further counterexample of $p$-adic algebraic groups was shown in \textbf{\cite{SL2QP}}. 

Work then moved instead towards describing the Ellis Group and considering the relationship to $G/G^{00}$. An explicit description of the Ellis Groups for groups admitting a compact-torsion free decomposition for $o$-minimal expansions of the reals was shown in \textbf{\cite{Jagiella}}. In \textbf{\cite{Yao}} they expanded this work by giving a description for Ellis Groups in this setting, but now interpreted in an arbitrary elementary extension.

Most recently, Jagiella \textbf{\cite{Jag2017}} works to remove $o$-minimal specific notions, demonstrating a way of computing Ellis Groups for $NIP$ groups that admit an $fsg$-definably amenable group decomposition, or for groups that admit a definably (extremely) amenable normal subgroup.

The study of metastable theories is shown extensively in \textbf{\cite{StabDom}} and \textbf{\cite{HR2017}}, with $ACVF$ as the motivating example. The work in \textbf{\cite{HR2017}} focuses on definable groups in metastable theories, demonstrating useful results on the existence of generic types in stably dominated groups. The motivation behind this work was to investigate whether analogues of Newelski's results in stable theories could be adapted to the metastable setting.

In this paper, we prove a group decomposition for $G = SL_2(\mathbb{C}((t)))$, and further demonstrate that this group is not definably amenable. We find the minimal subflows of $(\mathbb{C}((t)),+)$, $(\mathbb{C}((t))^*, \times)$ and the Borel subgroup of $SL_2(\mathbb{C}((t)))$, acting on their space of complete types. This leads to an explicit description of the Ellis Group of $(G, S_G(M))$, shown below. We remark that $SL_2(\mathbb{C}((t))^{00} = SL_2(\mathbb{C}((t)))$, and so the Ellis Group is not isomorphic to $G/G^{00}$, providing a negative answer to a question as to whether metastability is a suitable weakening of the Ellis Group conjecture of Newelski. Further, we hope to initiate study towards an explicit description of the Ellis Group for groups definable in metastable theories. 

\begin{Theorem} \label{Theorem_Main}
Let $M = \mathbb{C}((t))$, $G = SL_2$, and $B$ be the Borel subgroup of upper triangular matrices in $G$. Then the Ellis Group of $(G(M),S_G(M))$ is isomorphic to $B(M)/B(M)^{0}$.
\end{Theorem}

This paper is split into four main sections. In the first we recall the notions of definable topological dynamics as well as results on the model theory of $\mathbb{C}((t))$. In the second section, we prove $SL_2(\mathbb{C}((t)))$ is not definably amenable and establish a suitable group decomposition for $SL_2(\mathbb{C}((t))$. The third section contains the preliminary work involved in showing the invariant types and the minimal subflows of the additive and multiplicative groups of $\mathbb{C}((t))$ and of the Borel subgroup of $SL_2(\mathbb{C}((t))$. In the fourth and final section, we apply these results to build a minimal subflow of $(G(M),S_G(M))$ and provide an explicit description of the Ellis Group.

\section{Preliminaries} \label{Section_Preliminaries}

The majority of notation here is standard. $L$ will denote a language and $M$, $N$ will denote models unless otherwise stated, with $\bar{M}$ a global sufficiently saturated elementary extension of $M$. $M^{ext}$ will denote the Shelah expansion of $M$, obtained by adding a predicate for every externally definable subset of $M$. $G$, $H$ will denote groups, with lowercase $g$, $h$ denoting elements of $G$ and $H$ respectively. $p$, $q$ and $r$ are complete types. Most often $x,y$ will be variables with $a$,$b$,.. and $\alpha$,$\beta$,.. being parameters or field elements.

Given an $L$-structure $M$, we denote by $S(M)$ the space of complete types over $M$. If $G$ is a definable group ( more generally a definable set ), we use $S_G(M)$ to denote the space of complete types that contain the formula defining $G$. We use $G(M)$ to denote the interpretation of $G$ in $M$. $S_{G,ext}(M)$ is the space of complete external types concentrating on $G$.

By $\mathbb{C}((t))$ we mean the field of formal Laurent series with coefficients from $\mathbb{C}$; that is, the elements of $\mathbb{C}((t))$ are of the form $\sum\limits_{i=n}^{\infty} a_it^i$ for some $n \in \mathbb{Z}$. We now state some well known results about the field $\mathbb{C}((t))$, and more generally $k((t))$, where $k$ is algebraically closed. The key reference here is \textbf{\cite{Francoise}}. Our language will be the language of rings equipped with the following predicates; $P_n(x) \iff \exists y ( y^n = x )$ ; $N(x) \iff v(x) = 1$ and $x \mid y \iff v(x) \leq v(y)$. 

$\mathbb{C}((t))$ here is considered as a valued field with the $t$-adic valuation; given by $v(\sum\limits_{i = n}^{\infty} a_it^i) = n$, where $a_n$ is the first non-zero coefficient. We will use $\mathfrak{M}$ to denote the maximal ideal of $\mathbb{C}[[t]]$; that is, the elements $x \in \mathbb{C}((t))$ such that $v(x) > 0$. We note that $\mathbb{C}[[t]]$ and $\mathbb{C}$ are the valuation ring and residue field (resp.) of $\mathbb{C}((t))$.

We will use $res(a)$ to denote the residue of $a \in \mathbb{C}[[t]]$. In this setting, the angular component map $ac$ is not definable. However, for complete types which contain a formula $v(x) = \gamma$, we can access the angular component by considering $res(t^{-\gamma}x)$. Of course, this is only possible when $t^{-\gamma}$ is allowed as a parameter.

We recall the definition of metastable theories. It is shown in \textbf{\cite{StabDom}} that the theory of $\mathbb{C}((t))$ is metastable. First, a type $p$ (over $C$) is stably dominated if there exists a pro definable map (over $C$) $\alpha : p \rightarrow D$ with $D$ stable and stably embedded, such that for any $a \vDash p$ and tuple $b$, $\alpha(a) \forkindep_C D \cap dcl(b)$ implies $tp(b/C\alpha(a)) \vdash tp(b/Ca)$. A theory $T$ is metastable (over a sort $\Gamma$) if any set of parameters $C_0$ contained in a set $C$ (called a metastability basis), such that for any $a$ there exists a pro-definable map (over $C$) $\gamma_C p \rightarrow \Gamma$ with $tp(a/\gamma_C(a))$ stably dominated.

The model theory of fields $k((t))$, for $k$ algebraically closed, can be found in \textbf{\cite{Francoise}}, and we summarise some of the results here.

\begin{Fact} \textbf{\cite{Francoise}} \label{Fact_Language} \leavevmode
\begin{itemize}
 \item $\mathbb{C}((t))$ admits quantifier elimination in the language $(0, 1, + , \times, \hspace{1mm} \mid \hspace{1mm} , N ) \cup \{ P_n \text{ : } n \in \mathbb{N} \}$. 
 \item $\mathbb{C}((t))$ is NIP.
 \item The complete $1$-types over $M$ are definable. 
\end{itemize}
\end{Fact}

\begin{Definition} \label{Definition_CoHeirs}
Let $p$ be a type over some model $M$ of a theory $T$, and $q \in S(B)$ an extension of $p$ to $B \supset M$.
\begin{itemize}
 \item We call $q$ an \textbf{heir} of $p$ if for every $L(M)$-formula $\phi(x,y)$ such that $\phi(x,b) \in q$ for some $b \in B$ there is some $m \in M$ with $\phi(x,m) \in p$.
 \item We call $q$ a \textbf{coheir} of $p$ if $q$ is finitely satisfiable in $M$. 
\end{itemize}
\end{Definition}

\begin{Fact} \label{Fact_UniqueHeirsCoheirs} \textbf{\cite{SBook}}
Let $T$ have NIP and suppose that all complete types over $M$ are definable. Then every complete type $p$ over $M$ has a unique heir and unique coheir in $S(\bar{M})$. 
\end{Fact}

\begin{Definition} \label{Definition_G00}
 Let $G$ be a definable group. Then;
\begin{itemize}
 \item \textbf{$G^{0}_A$} is the intersection of all $A$-definable subgroups of $G$ with finite index in $G$, called the ``connected component (of $G$ over $A$)''. If $G^{0}_A = G^{0}_{\emptyset}$ for all $A$, then we say $G^{0}$ exists and drop the subscript notation.
 \item \textbf{$G^{00}_A$} is the smallest type-definable (with parameters from $A$) subgroup of $G$ of bounded index, sometimes called the ``type-definable connected component (of $G$ over $A$)''. If $G^{00}_A = G^{00}_{\emptyset}$ for all $A$, then we say $G^{00}$ exists and drop the subscript notation. 
 \item \textbf{Fact \cite{Shelah}} If $T$ is an $NIP$ theory and $G$ is definable in some model of $T$, then $G^{00}$ exists. 
\end{itemize}
\end{Definition}

Recall that a valued field $(K,v)$ is Henselian if, for any extension $L$ of $K$, $v$ extends uniquely to a valuation $w$ on $L$. It is well known that $\mathbb{C}((t))$ equipped with the $t$-adic valuation is a Henselian valued field.

\begin{Lemma}[\textbf{Hensel's Lemma}] \label{Lemma_Hensels} \leavevmode
Let $K$ be a Henselian valued field, complete with respect to some valuation $v$. Let $\mathcal{O}_K$ be the valuation ring of $K$ and $k$ the residue field of $K$. Let $f(x) \in \mathcal{O}_K[X]$.

Then the reduction $\bar{f}(x) \in k[x]$ has a simple root (that is, $a_0$ such that $\bar{f}(a_0) = 0$ and $\bar{f'}(a_0) \neq 0$), there exists a unique $a \in \mathcal{O}_K$ such that $f(a) = 0$ and the reduction (residue) $res(a) = a_0$. 
\end{Lemma}

\begin{Corollary}
Let $K$ be a Henselian valued field, complete with respect to some valuation $v$. Suppose that an element $a \in \mathcal{O}_K$ is in the coset $1 + \mathfrak{M}$. Then $a$ has $n^{th}$ roots for all $n \in \mathbb{N}$.
\end{Corollary}


We now establish some notation and preliminary results from topological dynamics. Let $G$ be a topological group and $X$ a compact (Hausdorff) topological space. We call a continuous action $G \times X \rightarrow X$ a $G$-flow, and write $(G,X)$. We will assume $G$ to be discrete, in which case a $G$-flow can be considered as an action of $G$ by homeomorphisms on a compact space $X$. By a subflow of $(G,X)$ we mean a flow $(G,Y)$ in which $Y \subset X$ is closed under the action of $G$, and further we note that $(G,X)$ will always have minimal and non-empty subflows. 

Since each $g \in G$ induces a homeomorphism of $X$, we consider the set of functions $\pi_g : X \rightarrow X$, where $\pi_g(x) = gx$. Recall the space $X^X$ of maps from $X$ to itself is a compact Hausdorff topological space. Hence we consider the closure of the set $\{ \pi_g \text{ : } g \in G \}$ in $X^X$ equipped with the product topology and obtain the enveloping semigroup $E(X)$, which is a semigroup of homeomorphisms of $X$ under the composition of mappings. Note that $E(X)$ is compact and we can further obtain another flow by considering the action of $G$ on $E(X)$. Ellis \textbf{\cite{Ellis}} proved the following relationships between the ideals of $E(X)$ and minimal subflows. 

\begin{Theorem} \label{Theorem_Ellis}
Let $(G,X)$ be a flow and construct the enveloping semigroup $E(X)$ as above. Let $J$ be the set of idempotents of $E(X)$. Then;
\begin{enumerate}[(i)]
 \item Minimal closed left ideals $I$ of the enveloping semigroup $(E(X), \circ)$ coincide with minimal subflows of $(G,X)$.
 \item If $I$ is a minimal closed left ideal, then $I \cap J \neq \emptyset$, and moreover for any $u \in I \cap J$, $(u \circ I, \circ)$ is a group, often called the ``Ellis Group''.
 \item The Ellis Groups obtained by varying the choice of $I$ or $u$ are isomorphic to each other, and so the Ellis Group is unique up to isomorphism.
\end{enumerate}
\end{Theorem}

We consider the above with some definable group $G$ and $S_G(M)$ the space of complete types concentrating on $G$. Then $(G(M),S_G(M))$ is a $G(M)$-flow with the action $* : G(M) \times S_G(M) \rightarrow S_G(M)$ given by $g * p = tp(g \cdot a / M)$ where $a$ realises $p$ over $M$ and $\cdot$ is the binary operation of $G$.

From this point onwards, we assume that $M$ a model of some $NIP$ theory.

In \textbf{\cite{Newelski2009}} Newelski shows that $E(S_{G,ext}(M))$ is isomorphic to $S_{G,ext}(M)$ as a $G(M)$-flow, where the action on $S_{G,ext}(M)$ is given by, for $p$, $q$ types in $S_{G,ext}(M)$, $p * q = tp(a \cdot b/M)$ where $a$ realizes $p$, $b$ realizes the unique heir of $q$ over $(M,a)$. Note that $S_{G,ext}(M)$ is homeomorphic to $S_G(M^{ext})$, and moreover if all types are definable over $M$, then $S_G(M^{ext}) = S_G(M)$.

In \textbf{\cite{PSC}} it is shown that many properties in $M^{ext}$ coincide with those in $M$. From this, we find that the semigroups $(E(S_G(M)), \circ)$ and $(S_G(M), *)$ are isomorphic. Hence, if all types over $M$ are definable, $(G,E(S_{G,ext}(M)))$ coincides with $(G,S_G(M))$. 

See for \textbf{\cite{SBook}} for a general but thorough reference for the following facts and definitions. 

\begin{Definition} \label{Definition_DefAmen}
Let $G$ be a group definable in a model $M$. A \textbf{Keisler Measure} $\mu$ (over $M$) is a finitely additive probability measure on formulas $\phi(x,m)$, for $m \in M$. 

Then $G$ is \textbf{Definably Amenable} if $G$ admits a global Keisler measure $\mu$ on the definable subsets of $G$ which is invariant under left (right) translation by elements of $G(\mathcal{U})$, where $\mathcal{U}$ is the global (monster) model. 
\end{Definition}

\begin{Definition} \label{Definition_fGeneric} 
A global type $p \in S_G(\mathcal{U})$ is left (right) $f$-generic over $A$ if no left (right) translate of $p$ forks over $A$. 
\end{Definition}

\begin{Fact} \textbf{\cite{SBook}} \label{Fact_DefAmenStab}
Let $M$ be a model and let $G$ be a group definable in $M$. Let $p \in S_G(\mathcal{U})$ be a global type.
\begin{enumerate}[(i)]
 \item If $G$ admits a global left invariant type, then $G$ is definable amenable.
 \item $p$ is $f$-generic if and only if $Stab(p) = G^{00}$.
 \item If $p$ is left (right) $f$-generic then $p$ is $G^{00}$-invariant. 
\end{enumerate}
\end{Fact}

By $G(M) * p$ we mean the set $\{ g * p \text{ : } \forall g \in G(M) \}$. For $X$ a topological space, the closure of $X$, denoted $cl(X)$, is the set of all limit points of $X$. 

\begin{Definition}
 A type $p \in S_G(M)$ is said to be \textbf{almost periodic} if $cl(G(M) * p)$ is a minimal subflow of $(G(M),S_G(M))$. Equivalently, $p \in S_G(M)$ is \textbf{almost periodic} if $p$ is in some minimal subflow of $(G(M),S_G(M))$. 
\end{Definition}

\begin{Fact} \label{Fact_Gorbit} 
Let $G$ be a definable group, let $p$ be a global type in $S_G(\bar{M})$, and let $G(\bar{M})$ be the interpretation of $G$ in $\bar{M}$. Then;
\begin{itemize}
 \item \textbf{\cite{Aus}} $cl(G(\bar{M}) * p)$ is $G(\bar{M})$-invariant.
 \item \textbf{\cite{Aus}} Let $X \subset S_G(\bar{M})$. $X$ is minimal if and only if $cl(G(\bar{M}) * p) = X$ for all $p \in X$. That is, a set is minimal exactly when it is the orbit closure of each of its points.
 \item $cl(G(\bar{M}) * p ) = S_G(\bar{M}) * p$.
 \item \textbf{\cite{2015APNY}} If $p$ is a global $f$-generic type, then $p$ is almost periodic and further $cl(G(\bar{M}) * p) = G(\bar{M}) * p$.  
\end{itemize}
\end{Fact}

\section{A Decomposition for $SL_2(\mathbb{C}((t)))$.}

From this point on, we will fix $M = \mathbb{C}((t))$ an $L$-structure, where $L$ is as in Fact \ref{Fact_Language}. We will use $M \prec \bar{M}$, with the domain of $\bar{M}$ denoted by $\mathbb{K}$. We will make no distinction between $M$ and $\mathbb{C}((t))$ or between $\bar{M}$ and $\mathbb{K}$. We will often use $\Gamma$ to mean the value group of $\mathbb{K}$.

\begin{Fact} \label{Fact_SL2C((t))NotDefAmen}
$SL_2(\mathbb{C}((t)))$ is not definably amenable.
\end{Fact}
\begin{proof}[\textbf{Proof}]
First, it is easy to check that $SL_2(\mathbb{C}((t))) = SL_2(\mathbb{C}((t)))^{00}$, using the fact that $SL_2(\mathbb{C}((t)))$ is simple. Assume for contradiction that $SL_2(\mathbb{C}((t)))$ is definably amenable. Then there is a global left $SL_2(\mathbb{C}((t)))^{00}$-invariant, and hence $SL_2(\mathbb{C}((t)))$-invariant, type $p(x)$. 

Let $x_{1,1}$ be the top left entry of a $2 \times 2$ matrix. Then $p(x) \vdash x_{1,1} \in C_i$ for some coset $C_i$ of $\mathbb{K}^{*0} = \bigcap\limits_n P_n(x)$.

Consider then translation $g$ of a realisation of $p(x)$, where $g = \begin{pmatrix} a & 0 \\ 0 & a^{-1} \end{pmatrix}$ for some $a \in \mathbb{C}((t))$. Then $gp(x) \vdash ax_{1,1} \in C_i$ if and only if $a$ is in the identity coset $\mathbb{K}^{*0}$. Clearly since $\mathbb{C}((t))$ is not algebraically closed we can easily find some suitable $a \notin \mathbb{K}^{*0}$ and see $gp(x) \neq p(x)$.

For completeness sake, if the $x_{1,1}$ entry of a realisation of $p(x)$ is $0$, since the determinant of the realisation is $1$, we see that the top right entry $x_{1,2} \neq 0$ and the above argument follows for the same $g$. Since $p$ is not $G^{00}$-invariant then $p$ is not an $f$-generic type and hence $SL_2(\mathbb{C}((t)))$ is not definably amenable by Fact \ref{Fact_DefAmenStab}.
\end{proof}

Recall that for a group $G$ a definable group, we denote by $G(M)$ the intepretation of $G$ in $M$.

Let $$H(M) = \left \{ \begin{pmatrix} 1 & 0 \\ \alpha & 1 \end{pmatrix} \text{ : } \alpha \in M \right \}$$ and let  $$B(M) = \left \{ \begin{pmatrix} \beta & \gamma \\ 0 & \beta^{-1} \end{pmatrix} \text{ : } \beta \in M^* \text{ and } \gamma \in M \right \}.$$ 

If we have some $L$-structure $\bar{M} = \mathbb{K}$, then $H(\bar{M})$ is isomorphic to $(\mathbb{K}, +)$. Similarly, $B(\bar{M})$ is isomorphic to a semidirect product of $(\mathbb{K}^*, \times)$ and $(\mathbb{K}, +)$. Both of these isomorphisms are definable.

Finally, we consider the subgroup $$\left \{ \begin{pmatrix} 1 & 0 \\ 0 & 1 \end{pmatrix} , \begin{pmatrix} -1 & 0 \\ 0 & -1 \end{pmatrix} , \begin{pmatrix} 0 & 1 \\ -1 & 0 \end{pmatrix} , \begin{pmatrix} 0 & -1 \\ 1 & 0 \end{pmatrix} \right \}.$$ We note that this is isomorphic to the cyclic group $\mathbb{Z}/4\mathbb{Z}$ and will use $\mathbb{Z}/4\mathbb{Z}$ to denote the above subgroup. 

\begin{Proposition} \label{Proposition_C((t))Decomp}
Every element of $G(M) = SL_2(\mathbb{C}((t)))$ can be expressed as a product of elements from $\mathbb{Z}/4\mathbb{Z}$, $H(M)$ and $B(M)$.
\end{Proposition}
\begin{proof}[\textbf{Proof}]
Let $g = \begin{pmatrix} x_1 & x_2 \\ x_3 & x_3 \end{pmatrix}$ be an arbitrary matrix in $G(M)$. Assume that $x_1 \neq 0$ and let $\beta = x_1 \neq 0$, $\gamma = x_2$, $\alpha = x_3(x^{-1}_1)$. Then;
\begin{center}
$ \begin{pmatrix} 1 & 0 \\ 0 & 1 \end{pmatrix}
  \begin{pmatrix} 1 & 0 \\ \alpha & 1 \end{pmatrix} \begin{pmatrix} \beta & \gamma \\ 0 & \beta^{-1} \end{pmatrix} = \begin{pmatrix} \beta & \gamma \\ \beta\alpha & \beta^{-1} + \alpha\gamma \end{pmatrix} = \begin{pmatrix} x_1 & x_2 \\ x_3 & x_4 \end{pmatrix}$ 
\end{center}

It remains to show that we can obtain matrices where $x_1 = 0$. Choose $z \in \mathbb{Z}/4\mathbb{Z}$ to be the matrix $\begin{pmatrix} 0 & -1 \\ 1 & 0 \end{pmatrix}$, and assume $x_1=0$. Let $\alpha = 0$, $\beta = x_3$ and $\gamma = x_4$. Then;

\begin{center}
$ \begin{pmatrix} 0 & -1 \\ 1 & 0 \end{pmatrix} \begin{pmatrix} 1 & 0 \\ \alpha & 1 \end{pmatrix} \begin{pmatrix} \beta & \gamma \\ 0 & \beta^{-1} \end{pmatrix} = \begin{pmatrix} -\alpha\beta & -(\beta^{-1} + \alpha\gamma) \\ \beta & \gamma \end{pmatrix} = \begin{pmatrix} x_1 & x_2 \\ x_3 & x_4 \end{pmatrix}$ 
\end{center}

Hence for any given arbitrary matrix we can solve for some element $z$, together with a choice of $\alpha$, $\beta$ and $\gamma$ and decompose that matrix as above.
\end{proof}

Note that there are multiple ways (at most $4$) to decompose an element of $G(M)$ in this way. In \textbf{\cite{SL2QP}}, they use an Iwasawa-like decomposition of $SL_2(\mathbb{Q}_p)$ and make use of the fact that $SL_2(\mathbb{Z}_p)$ is compact. As $SL_2(\mathbb{C}[[t]])$ is not even locally compact, we could not adopt a similar approach using the Iwasawa-like decomposition of $SL_2(\mathbb{C}((t)))$. 

\section{Minimal Flows of $\mathbb{G}_a$, $\mathbb{G}_m$, and Borel Subgroups.}

Recall the language used in Fact \ref{Fact_Language}, and that $P_n(x) \iff \exists y ( y^n = x )$. Then $P_n(x)$ and $x \neq 0$ determine a finite index subgroup of $\mathbb{K}^*$. It is clear that the type $\bigwedge\limits_{n} P_n(x)$ determines the connected component $\mathbb{K}^{*0}$ of the multiplicative group $(\mathbb{K}^*, \times)$. We will use $C_i$ to denote an arbitrary coset of the connected component $\mathbb{K}^{*0}$, with $C_0$ denoting the identity coset $\mathbb{K}^{*0}$ itself.          

\begin{Lemma} \label{Lemma_1TypesC((t))}
 The complete 1-types over $M = (\mathbb{C}((t)), +, \times)$ are precisely the following;
\begin{enumerate}[(a)]
 \item The (realized) types $tp(a/M)$ for each $a \in \mathbb{C}((t))$.
 \item For each $a \in \mathbb{C}((t))$ and coset $C$ of $\mathbb{K}^{*0}$, the type $p_{a,C}$ determined by \\ $\{ v(x-a) > n \text{ : } \forall n \in \mathbb{N} \}$ and $(x - a) \in C$.
 \item For each coset $C$ of $\mathbb{K}^{*0}$, the type $p_{ \infty, C}$ determined by \\
$\{ v(x) < n \text{ : } \forall n \in \mathbb{Z} \}$ and $x \in C$.
 \item For each $a \in \mathbb{C}[t]$, the type $p_{a,n,trans}$ determined by the formulas \\ $ v(x - a) = n $, $deg(a) < n$ for some $n \in \mathbb{Z}$ and \\ $\{ f(res((x-a)t^{-n})) \neq 0 \text{ : } f \in (\mathbb{C})[x] \}$. \\ If $a = 0$ then, we can drop the $deg(a) < n$ from the description.
\end{enumerate}
\end{Lemma}
\begin{proof}[\textbf{Proof}]
The classification of 1-types over $\mathbb{C}((t))$ can be found in \textbf{\cite{Francoise}}, though we rename them here and observe that types of kind $(a)$ are the immediate types, kinds $(b)$ and $(c)$ are the valuational types, and types of kind $(d)$ are the residual types. 

The proof that the types of kind $(b)$ and $(c)$ are complete follows similarly to the proof in \cite{SL2QP}. We prove types of kind $(d)$ directly, though one can also observe that these types are translations of a type shown to be complete in \cite{Francoise}.

Let $p_{b,n,trans}$ be as above. Let $x_0 \vDash p_{b,n,trans}$. Then $x_0 = b + x_1 = a + \alpha t^n + ...$, where $\alpha$ is transcendental over $\mathbb{C}$.

Then by QE, we consider polynomials $f$ of $\mathbb{C}((t))[x]$. We may assume $f$ is not a constant polynomial. Hence; 
\begin{align*}
f(x_0) & = a_0 + a_1x_0 + a_2x_0^2 + ... + a_mx_0^m \\ 
       & = a_0 + a_1(b+x_1) + a_2(b+x_1)^2 + ... + a_m(b+x_1)^m \\
       & = ( a_0 + a_1b + a_2b^2 + ... + a_mb^m ) + (a_1x_1 + a_2(2bx_1 + x_1^2) + .... + a_mx_1^m) \\
       & = f(b) + c_1x_1 + c_2x_1^2 + c_3x_1^3 + ... c_mx_1^m \\
       & = f(b) + g(x_1)
\end{align*}
Where the coefficients $c_i$ are elements of $\mathbb{K}^{*0}$. This is possible since $b \in \mathbb{C}((t))$. Remember that the angular component is not definable in this setting, but since each $c_ix_1^i$ has some valuation $z \in \mathbb{Z}$ and we allow parameters from $M$, we can instead consider $res(c_ix_1^it^{-z})$. Hence, since $x_1$ is transcendental over $\mathbb{C}$, $res(c_ix_1^it^{-z})$ is transcendental over $\mathbb{C}$.

Also note that since $\mathbb{C}$ is algebraically closed, $x_1^i \notin res(\mathbb{K})$ for any $i$. We can express $g(x_1) = d_0\delta_0 + d_1\delta_1 + d_2\delta_2 +... $ with $d_i \in \mathbb{C}((t))$ and $\delta_i$ transcendental over $\mathbb{C}$. Using this, we see $P_n(g(x_1)) \iff P_n(d_0\delta_0) \iff P_n(d_0)$ since $\delta_0$ is transcendental over $\mathbb{C}$. Further, in some expansion $\mathbb{K}$ with residue field $\delta_0 \in acl(res(\mathbb{K}))$, $\delta_0$ has $n^{th}$ roots for all $n$, and so $\vDash P_n(\delta_0)$ for all $n$.

Hence $P_n(g(x_1)) \iff P_n(d_0)$. Further, since $f(b)$ is an element of $\mathbb{C}$, we see $P_n(f(x_0)) \iff P_n(f(b) + g(x_0)) \iff P_n(f(b) + d_0)$. Hence $P_n(f(x_0))$ is determined.

We can determine $N(f(x_0))$ similarly by considering the valuation of $g(x_1)$. Hence $P_n(f(x_0))$ and $N(f(x_0))$ are determined and $p_{b,n,trans}$ is a complete type as required. 
\end{proof}

\begin{Corollary}
Every (left) $\mathbb{K}^*$-translate of of the global heir of $p_{0,C}$ is definable over $M$.
\end{Corollary}
\begin{proof}[\textbf{Proof}]
Let $a \in \bar{M}^*$ and $x_0 \vDash p_{0,C}$
Suppose $P_n(f(x)) \in ap_{0,C}$. Then $ax_0 \vDash P_n(f(x)) \iff a^{-1}P_n(f(x)) \in p_{0,C}.$ Since $P_n(\mathbb{K}^*)$ has finite index in $\mathbb{K}^*$, $\exists b \in M^*$ such that $a^{-1}P_n(\mathbb{K}^*) = bP_n(\mathbb{K}^*)$. Hence $P_n(f(x)) \in ap_{0,C} \iff bP_n(f(x)) \in p_{0,C}$. As $p_{0,C}$ is $\emptyset$-definable, $b$ is $\mathbb{C}((t))$-definable and so $ap_{0,C}$ is definable over $M$ as required.  
\end{proof}

We denote by $S_{\mathbb{G}_a}(M)$ the space of complete types concentrating on $\mathbb{G}_a$, where $\mathbb{G}_a(M) = (\mathbb{C}((t)),+)$, and so $S_{\mathbb{G}_a}(M)$ is a flow under the additive group action.  

\begin{Proposition} \leavevmode \label{Proposition_MinSubAdd}
 \begin{enumerate}[(i)]
  \item The types $p(x) \in S_{\mathbb{G}_a}(M)$ of kind $(c)$ are definable generic types of $(\mathbb{G}_a, +)$. Moreover, the global heir of $p_{\infty,C}$ is invariant under the action of $(\mathbb{K},+)$ for any coset $C$ of $\mathbb{K}^{*0}$.
  \item The types $p_{\infty, C}$ are 1-point minimal subflows of $(\mathbb{G}_a(M),S_{\mathbb{G}_a}(M))$.
  \item The global heirs of the types of kind $(c)$ are precisely the global (strongly) $f$-generics of $(\mathbb{K},+)$ and are all definable and invariant under $(\mathbb{K},+)$.
  \item $\mathbb{K}^{00}$ = $\mathbb{K}^0 = \mathbb{K}$. 
 \end{enumerate}
\end{Proposition}
\begin{proof}[\textbf{Proof}] \leavevmode
\begin{enumerate}[(i)]
 \item Suppose that $a \in \mathbb{K}$ and $\beta \vDash p_{\infty, C}$.
Since $v(\beta) = \alpha < \Gamma$, and $v(a) = c \in \Gamma$, we have $v(\beta) < v(a)$ and hence $v(a + \beta) < \Gamma$. It remains to show that $a + \beta \in P_n(\mathbb{K})$. Since $\beta^{-1} \in P_n(\mathbb{K})$, we see $a + \beta \in P_n(\mathbb{K}) \iff \beta^{-1}(a + \beta) = 1 + a \beta^{-1} \in P_n(\mathbb{K})$.

Moreover, $v(\beta) = -v(\beta^{-1})$, and $v(\beta^{-1}) > \Gamma$ and so $v(a \beta^{-1}) > \Gamma$, we have that $1 + a \beta^{-1} \in 1 + \mathfrak{M}$, and so by the corollary to Hensel's Lemma, $1 + a \beta^{-1}$ has $n^{th}$ roots in $\mathbb{K}$ and hence $a + \beta \vDash p_{\infty, C}$. Since this type is $\mathbb{K}$-invariant, in general, it must be $\mathbb{K}^{00}$-invariant and hence is $f$-generic.

\item This follows from $(i)$. Let $q \in S_{(\mathbb{C}((t)),+)}(M)$ and consider $q * p_{\infty,C} = tp(a + \beta/M)$, where $a \vDash q$ and $\beta \vDash p_{\infty,C}|_{M,a}$. Then from above we have $tp(a + \beta / M) = tp(\beta / M)$, and hence is a subflow of $S_{\mathbb{G}_a}(M)$ under the action of $(\mathbb{K},+)$. Minimality follows trivially, since $p_{\infty,C}$ is a singleton there can be no properly contained non-empty subflow.

\item From $(ii)$, $p_{\infty,C}$ is $(\mathbb{K},+)$-invariant, and so in particular must be $(\mathbb{K}^{00}, +)$-invariant, and so by Fact \ref{Fact_DefAmenStab}, $p_{\infty,C}$ is $f$-generic. We note that the choice of $C$ was arbitrary in $(ii)$, and so all types of kind $(c)$ are (strongly) $f$-generic as required.

\item Since $p_{\infty,C}$ is a global $f$-generic type, by Fact \ref{Fact_DefAmenStab} $\mathbb{K}^{00}$ is precisely the stabilizer of the type $p_{\infty,C}$ ; $Stab(p_{\infty,C}) = \{ g \text{ : } gp = p \} = \mathbb{K}$.
\end{enumerate}
\end{proof}

Similarly, we denote by $S_{\mathbb{G}_m}(M)$ the space of complete types concentrating on $\mathbb{G}_m$, where $\mathbb{G}_m(M) = (\mathbb{C}((t))^*, \times)$, and so $S_{\mathbb{G}_m}(M)$ is a flow under multiplication of non-zero field elements. 

\begin{Proposition} \leavevmode \label{Proposition_MinSubMult}
\begin{enumerate}[(i)]
 \item The types $P_{0} = \{ p_{0,C} \text{ : } v(x) > n \hspace{1mm} \forall n \in \Gamma$ \text{ and } $C$ some coset of $(\mathbb{K}^*)^0 \}$ form a minimal subflow of $(\mathbb{G}_m, S_{\mathbb{G}_m}(M))$.
 \item  The types $P_{\infty} = \{ p_{\infty, C} \text{ : } v(x) < n \hspace{1mm} \forall n \in \Gamma$ and $C$ some coset of $\mathbb{K}^*)^0 \}$ form a minimal subflow of $(\mathbb{G}_m, S_{\mathbb{G}_m}(M))$.
 \item The global heirs of the these types are precisely the global (strongly) $f$-generics of $(\mathbb{K}^*,\times)$ and are all definable. The orbit of each such type of $\mathbb{K}^*$ is closed.
 \item The type-definable connected component $\mathbb{K}^{*00}$ coincides with the definable connected component $\mathbb{K}^{*0}$.
\end{enumerate}
\end{Proposition}
\begin{proof}[\textbf{Proof}] \leavevmode
\begin{enumerate}[(i)]
\item To show $P_0$ is a minimal subflow, we show it is precisely the $S_{\mathbb{G}_m}(M)$-orbit of a type $p_{0,C_0}$.

Let $q \in S_{\mathbb{G}_m}(M)$ with $a$ realising $q$ and $\alpha$ realise the heir of $p_{0,C_0}$ over $(M,a)$. Then $q * p_{0,C_0} = tp(a\alpha/M)$. Since $v(a\alpha) = v(a) + v(\alpha) > \Gamma$, then $tp(a\alpha/M)$ must be a type of kind $(b)$, with $a\alpha$ infinitesimally close to $0$. Hence, $tp(a\alpha/M) = p_{0,C_i}$ for $C_i$ some coset of $\mathbb{K}^{*0}$. Further, since $\alpha$ is an element of the identity coset $C_0$, we have $a\alpha \in C_i \iff a \in C_i$.

However, the choice of $q$ (and $a$) was arbitrary. In particular, $a$ could lie in any coset $C_i$, and so the $S_{\mathbb{G}_m}(M)$-orbit of $p_{0,C_0}$ is \\ $P_0 = \{ p_{0,C_i} \text{ : } C_i \text{ a coset of } \mathbb{K}^{*0} \}$. It is clear that this is the orbit-closure of any type $p_{0,C_i} \in P_0$.

\item To show $P_{\infty}$ is a minimal subflow, we show it is precisely the $S_{\mathbb{G}_m}(M)$-orbit of a type $p_{\infty,C_0}$.

Let $q \in S_{\mathbb{G}_m}(M)$ with $a$ realising $q$ and $\alpha$ realise the heir of $p_{\infty,C_0}$ over $(M,a)$. Then $q * p_{\infty,C_0} = tp(a\alpha/M)$. Since $v(a\alpha) = v(a) + v(\alpha) < \Gamma$, then $tp(a\alpha/M)$ must be of kind $(c)$; that is, $tp(a\alpha/M) = p_{\infty,C_i}$ for $C_i$ some coset of $\mathbb{K}^{*0}$. Again, since $\alpha$ is an element of the identity coset $C_0$, we have $a\alpha \in C_i \iff a \in C_i$.

However, the choice of $q$ (and $a$) was arbitrary. In particular, $a$ could lie in any coset $C_i$, and so the $S_{\mathbb{G}_m}(M)$-orbit of $p_{\infty,C_0}$ is \\ $P_{\infty} = \{ p_{\infty,C_i} \text{ : } C_i \text{ a coset of } \mathbb{K}^{*0} \}$ as required.

\item We recall that $\mathbb{K}^{*0} = \bigcap\limits_{n \in \mathbb{N}} P_n(x)$, and from $(i)$ and $(ii)$ we see that types of kind $(b)$ and $(c)$ are $\mathbb{K}^{*0}$-invariant. Further, $\mathbb{K}^{*00} \subseteq \mathbb{K}^{*0}$, and so types of kind $(b)$ and $(c)$ are in particular $\mathbb{K}^{*00}$-invariant. Hence by Fact \ref{Fact_DefAmenStab} these types are $f$-generic.

\item From Fact \ref{Fact_DefAmenStab}, we see that $Stab(p_{\infty,C_0}) = \mathbb{K}^{*00}$. Note from earlier that for an element $x \in C_0$, any translate $ax \in C_0 \iff a \in C_0$, and so $\{ g \text{ : } gp = p \}$ is precisely when $g \in C_0$. But $C_0$ is $\mathbb{K}^{*0}$, and so we also see that $Stab(p_{\infty,C_0}) = \mathbb{K}^{*0}$.

Hence $\mathbb{K}^{*0} = \mathbb{K}^{*00}$.
\end{enumerate}
\end{proof}

We now consider the Borel subgroup, $B(\bar{M})$, of upper triangular matrices. We will often associate the matrix $\begin{pmatrix} b & c \\ 0 & b^{-1} \end{pmatrix} \in B(\bar{M})$ with the pair $(b,c)$ where $b \in \mathbb{K}^*$ and $c \in \mathbb{K}$. 

\begin{Lemma} \label{Lemma_B00}
$B(\bar{M})^{00} = B(\bar{M})^{0} \cong \{ (b,c) \text{ : } b \in \mathbb{K}^{*0}, c \in \mathbb{K} \}$.
\end{Lemma}
\begin{proof}[\textbf{Proof}] \leavevmode
Consider the following mapping; 
\begin{center}
 \begin{align*}
 \pi :  B(\bar{M}) & \rightarrow \mathbb{K}^* \\
             (b,c) & \mapsto b
 \end{align*}
\end{center}
With $Ker(\pi) = (\mathbb{K}, +)$. Then it is clear that $\pi : B(\mathbb{K}^{00}) \rightarrow \mathbb{K}^{*00}$ with Kernel isomorphic to $(\mathbb{K}^{00}, +)$. Using the results of Propositions \ref{Proposition_MinSubAdd} and \ref{Proposition_MinSubMult} that $\mathbb{K}^{*00} = \mathbb{K}^{*0}$ and $(\mathbb{K}^{00}, +) = (\mathbb{K}, +)$, we obtain $B(\bar{M})^{00} = B(\bar{M})^{0} = \{ (b,c) \text{ : } b \in \mathbb{K}^{*0}, c \in \mathbb{K} \}$. 
\end{proof}

Recall that $C_0$ denotes $(\mathbb{K}^*)^0$, $\bar{p}_{0,C_{0}}$ is a global $f$-generic of $S_{\mathbb{G}_a}(\bar{M})$ and that $\bar{p}_{\infty,C_0}$ is a global $f$-generic type in $S_{\mathbb{G}_m}(\bar{M})$. Let $\beta$ realise $\bar{p}_{0,C_0}$ and $\gamma$ realize the heir of $\bar{p_{\infty,C_0}}$ over $(\bar{M}, \beta)$. 

Consider then the pairs $(\beta,0)$ and $(1,\gamma)$ and we identify these pairs with the types $tp((\beta,0)/\bar{M})$ and $tp((1,\gamma)/\bar{M}, \beta)$ of the corresponding matrix. Then $p_{0,C_0} * p_{\infty,C_0} = tp((\beta,0)/\bar{M}) * tp((1,\gamma)/\bar{M}, \beta) = tp((\beta,\gamma\beta)/\bar{M})$.

Let $\bar{p_0} = tp((\beta,\gamma)/\bar{M}) \in S_B(\bar{M})$, and so by $p_0$ we mean the restriction of this type to $M$. 

\begin{Lemma} \label{Lemma_p0}
 $\bar{p_0} \in S_B(\bar{M})$ is a (strong) $f$-generic of $B(\bar{M})$, and moreover every left $B(\bar{M})$-translate is definable over $M$ (i.e. definable over $\mathbb{C}((t))$).
\end{Lemma}
\begin{proof}[\textbf{Proof}]
We show $f$-genericity by proving that $\bar{p_0}$ is $B(\bar{M})^{00}$-invariant.
Let $(b,c) \in B(\bar{M})^{00}$, which by Lemma \ref{Lemma_B00}, means $b \in (\mathbb{K}^*)^0$ and $c \in \mathbb{K}$. Since the operation here is matrix multiplication, we note that $(b,c)(\beta,\gamma) = (b\beta, b\gamma + c\beta^{-1})$.

We want to show that $tp((b\beta, b\gamma + c\beta^{-1})/\bar{M}) = tp((\beta,\gamma)/\bar{M})$.

It is equivalent to show that $tp(b\beta/\bar{M}) = tp(\beta/\bar{M})$ and that $tp(b\gamma + c\beta^{-1}/\bar{M}, b\beta)$ is an heir of $tp(\gamma/\bar{M})$.

As $b \in \mathbb{K}^{*0}$ we have that $tp(b\beta/\bar{M}) = tp(\beta/\bar{M})$. Then since $b\beta \equiv_{\bar{M}} \beta$, $\gamma$ must also realise the heir of $p_{\infty,C_0}$ over $\bar{M}, b\beta$.

Since $p_{\infty,C_0}$ is invariant under multiplication by elements of $\mathbb{K}^{*0}$, 
$tp(b \gamma / \bar{M}, b\beta) = tp(\gamma / \bar{M}, b\beta)$. 
Moreover, as $v(\gamma) < dcl(M,\beta) \cap \Gamma$, $tp(b \gamma + c\beta^{-1} / \bar{M}, b\beta) = tp(b\gamma / \bar{M}, b\beta)$ and so $tp(b \gamma + c\beta^{-1} / \bar{M}, b\beta)$ is an heir of $tp(\gamma/\bar{M})$.

Since $b\beta$ realises $p_{0,C_0}$ and $(b\beta)^{-1}(b \gamma + c \beta^{-1})$ realises the unique heir of $p_{\infty,C_0}$ over $(M,b\beta)$, we have that $p_{0,C_0} * p_{\infty,C_0} = tp((b\beta, b\gamma + c\beta^{-1}) / \bar{M}) = tp((\beta,\gamma)/\bar{M})$.

Then $\bar{p_0}$ is a $B(\bar{M})^{00}$-invariant type of $S_B(\bar{M})$ and hence $f$-generic by Fact \ref{Fact_DefAmenStab}.

Finally, since $tp(\beta/\bar{M})$ is definable over $M$, and $tp(\gamma / \bar{M}, \beta)$ is the heir of $p_{\infty,C_0}$, which is also definable over $M$, we have that $tp((\beta,\gamma)/\bar{M})$ is definable over $M$. It is clear using the above argument that every left $B(\bar{M})$-translate of $\bar{p_0}$ is definable over $M$. 
\end{proof}

\begin{Proposition} \leavevmode \label{Proposition_MinSubBorel}
\begin{enumerate}[(i)]
 \item The $B(\bar{M})$-orbit of $\bar{p_0}$ is closed and is a minimal $B(\bar{M})$-subflow of $S_B(\bar{M})$.
 \item The restriction of $\bar{\mathcal{J}}$ to $M$, denoted $\mathcal{J}$, is a subgroup of $(S_B(M), *)$, is isomorphic to $B(\bar{M})/B(\bar{M})^{0}$ and hence is the Ellis Group of the flow $(B(M), S_B(M))$.
\end{enumerate}
\end{Proposition}
\begin{proof}[\textbf{Proof}] \leavevmode
\begin{enumerate}[(i)]
 \item The fact that the orbit is closed follows from Lemma 1.15 of \textbf{\cite{2015APNY}}, and it is well known that a non-empty set is a minimal flow if and only if it is the orbit closure of each of its points, a proof of which can be found in \textbf{\cite{Aus}}.

 \item First, we note that $p_0$ is itself contained in $S_{B^{0}}(M)$, and since $p_0$ is $B(\bar{M})^{0}$-invariant by Lemma \ref{Lemma_p0}, we have that $p_0$ is idempotent. 

 That subflows are preserved under restrictions is a consequence of Proposition 5.4 of \textbf{\cite{2015APNY}}. Since $\bar{\mathcal{J}}$ is the minimal subflow of $(B(\bar{M}),S_B(\bar{M}))$, the restriction $\mathcal{J}$ is a minimal subflow of $(B(M),S_B(M))$. We can then form the Ellis group of $(B(M),S_B(M))$, which is $(p_0 * \mathcal{J}, *)$.

We now show that $p_0 * \mathcal{J} = \mathcal{J}$. Clearly, $p_0 * \mathcal{J} \subseteq \mathcal{J}$, as $\mathcal{J}$ is a minimal subflow of $S_B(M)$. Let $p_i \in \mathcal{J}$. We claim that $p_0 * p_i = p_i$. 

\textbf{Claim:} $p_0 * p_i = p_i$.
\begin{proof}[\textbf{Proof of Claim}]
Let $(b,c)$ realise $p_0$ and $(\beta,\gamma)$ realise the heir of $p_i$ over $(M, (b,c))$. Then we want to show that $tp((b\beta,b\gamma + c\beta^{-1})/M) = tp((\beta,\gamma)/M)$. 

The valuational arguments in Lemma \ref{Lemma_p0} carry over, namely that if $\beta$ is negatively infinitely valued then so is $b\beta$, and that $v(\gamma) < dcl(M,\beta) \cap \Gamma$. It remains to prove that $b\beta$ lies in the same coset as $\beta$, and that $b\gamma + c\beta^{-1}$ lies in the same coset as $\gamma$.

Since $\mathbb{K}^{*0}$ acts as the identity on $\mathbb{K}^*/\mathbb{K}^{*0}$, and $b \in \mathbb{K}^{*0}$, we know that $b\beta$ lies in the same coset as $\beta$. 

For $b\gamma + c\beta^{-1}$ this is not as clear. Instead observe that $\gamma + b^{-1}c\beta^{-1}$ is in the same coset as $\gamma$ if and only if $\gamma^{-1}(\gamma + b^{-1}c\beta^{-1}) = 1 + \gamma^{-1}b^{-1}c\beta^{-1} \in \mathbb{K}^{*0}$. Since $v(\gamma^{-1})$ is infinite over $(M,b,c,\beta)$, we see that $1 + \gamma^{-1}b^{-1}c\beta^{-1} \in 1 + \mathfrak{M}$. 

Hence by the corollary to Hensel's Lemma (\ref{Lemma_Hensels}), $1 + \gamma^{-1}b^{-1}c\beta^{-1}$ has $n^{th}$ roots for all $n$ and so lies in $\mathbb{K}^{*0}$. Hence $\gamma + b^{-1}c\beta^{-1}$ lies in the same coset as $\gamma$. Finally, since $b \in \mathbb{K}^{*0}$, $b(\gamma + b^{-1}c\beta^{-1}) = b\gamma + c\beta^{-1}$ also lies in the same coset as $\gamma$ as required. Hence $p_0 * p_i = p_i$ for all $p_i \in \mathcal{J}$.
\end{proof}

So $p_0 * \mathcal{J} = \mathcal{J}$, with $p_0$ acting as identity we see $(p_0 * \mathcal{J}, *) = (\mathcal{J}, *)$.

From this, we obtain the following map;
\begin{align*}
\pi : \mathcal{J} & \rightarrow B(\bar{M})/B(\bar{M})^{0} \\ 
      tp(t/M) & \mapsto tB(\bar{M})^{0}
\end{align*}
We show that $\pi$ is an isomorphism. First, we show $\pi$ is a group homomorphism.

$p_0 = tp(t_0/M)$ where $t_0$ is a matrix in $B(\bar{M})^{0}$, and so clearly $t_0B(\bar{M})^{0} = B(\bar{M})^{0}$, which is the identity element of the quotient group.

Since $\mathcal{J}$ is a group, let $p_i \in \mathcal{J}$ with inverse $p_{i}^{-1}$. Then $p_i$ is realised by some $t_i$, and so $\pi(p_i) = t_i(B(\bar{M})^{0})$.

The heir of $p_{i}^{-1}$ over $(M,t_i)$, and so in particular $p_{i}^{-1}$ itself, is realised by some $s_i$, and so $\pi(p_{i}^{-1}) = s_i(B(\bar{M})^{0})$. Note we are not claiming $s_i$ is the inverse of $t_i$ in $G$, just that $s_i$ is a realisation of the inverse of $p_i$ in $\mathcal{J}$.

We claim that $s_i(B(\bar{M})^{0}) \cdot t_i(B(\bar{M})^{0}) = B(\bar{M})^{0}$.

From coset multiplication we have $s_i(B(\bar{M})^{0}) \cdot t_i(B(\bar{M})^{0}) = (t_i s_i)B(\bar{M})^{0}$. Then as $p_i^{-1} * p_{i} = p_0$, we see $s_i t_i \in B(\bar{M})^{0}$, and so $s_i(B(\bar{M})^{0}) \cdot t_i(B(\bar{M})^{0}) = B(\bar{M})^{0}$.

Hence $\pi((p_i)^{-1}) = ((t_i)(B(\bar{M})^{0}))^{-1}$ as required, and so $\pi$ is a group homomorphism. 

It is easy to see that $\pi$ is bijective. Note $\mathcal{J}$ is a section of $B(\bar{M})/B(\bar{M})^{0}$. That is to say $\pi$ is surjective since, for every coset $t(B(\bar{M})^{0})$, we can associate a type $p_i \in \mathcal{J}$ with $t' \in t(B(\bar{M})^0)$ and $t' \vDash p_i$ such that $\pi(p_i) = t(B(\bar{M})^{0})$. Injectivity follows from the definition of $\mathcal{J}$, observing that each type in $\mathcal{J}$ is determined uniquely by a coset of $\mathbb{K}^{*0}$.
\end{enumerate}
\end{proof}

\section{The Minimal Subflow of $(G(M),S_G(M))$.}

As before, $p_i$ will denote a type in $\mathcal{J}$, where $p_i$ specifies in which coset $C_i$ of $\mathbb{K}^{*0}$ the realisation of $p_i$ lies. We will again use the notation in \ref{Lemma_1TypesC((t))} for the valuational types, with $p_{\infty,C_0}$ the minimal subflow of $(\mathbb{G}_a, S_{\mathbb{G}_a}(M))$. 

We will often associate some $h \vDash p_{\infty,C_0}$ with the matrix $\begin{pmatrix} 1 & 0 \\ \alpha & 1 \end{pmatrix}$, and likewise some $t \vDash p$ with the matrix $\begin{pmatrix} \beta & \gamma \\ 0 & \beta^{-1} \end{pmatrix}$. We will not distinguish between $z \in \mathbb{Z}/4\mathbb{Z}$ and the type determined by the formula $x = z$. As before, $\mathbb{K}$ will denote some elementary extension of $\mathbb{C}((t))$ with $C_0$ denoting the identity coset $\mathbb{K}^{*0}$ itself.

We approach this by attempting to build a minimal subflow around an idempotent element which lies in the $*$-product of the minimal subflows of $H$ and $B$.

\begin{Proposition} \label{Proposition_C((t))Idemp}
Let $p_0 \in \mathcal{J}$ as in Lemma \ref{Lemma_p0} and $p_{\infty,C_0}$ a minimal subflow of $S_{\mathbb{G}_a}(M)$. Then the type $p_{\infty,C_0} * p_0$ is an idempotent element of $(S_G(M), *)$.
\end{Proposition}
\begin{proof}[\textbf{Proof}]
To show this is an idempotent, we need to show $(p_{\infty,C_0} * p_0) * ( p_{\infty,C_0} * p_0) = p_{\infty,C_0} * p_0$.

Let $h_0$ realise $p_{\infty,C_0}$, let $t_0$ realise the heir of $p_0$ over $(M,h_0)$, let $h$ realise the heir of $p_{\infty,C_0}$ over $(M,h_0,t_0)$ and let $t$ realise the heir of $p_0$ over $(M, h_0,t_0,h)$. Then $(p_{\infty,C_0} * p_0)^2 = tp(h_0t_0ht/M)$. Then;

\begin{align*}
h_0t_0ht & = \begin{pmatrix} 1 & 0 \\ a & 1 \end{pmatrix} \begin{pmatrix} b & c \\ 0 & b^{-1} \end{pmatrix} \begin{pmatrix} 1 & 0 \\ \alpha & 1 \end{pmatrix} \begin{pmatrix} \beta & \gamma \\ 0 & \beta^{-1} \end{pmatrix} \\
& = \begin{pmatrix} 1 & 0 \\ a & 1 \end{pmatrix} \begin{pmatrix} b + c\alpha & c \\ b^{-1}\alpha & b^{-1} \end{pmatrix} \begin{pmatrix} \beta & \gamma \\ 0 & \beta^{-1} \end{pmatrix} \\
& = \begin{pmatrix} 1 & 0 \\ a & 1 \end{pmatrix} \begin{pmatrix} 1 & 0 \\ \frac{b^{-1}\alpha}{b + c\alpha} & 1 \end{pmatrix} \begin{pmatrix} b + c\alpha & c \\ 0 & (b+c\alpha)^{-1} \end{pmatrix} \begin{pmatrix} \beta & \gamma \\ 0 & \beta^{-1} \end{pmatrix} \\
& = \begin{pmatrix} 1 & 0 \\ a + \frac{b^{-1}\alpha}{b + c\alpha} & 1 \end{pmatrix} \begin{pmatrix} \beta(b + c\alpha) & \gamma(b+c\alpha) + c\beta^{-1} \\ 0 & \beta^{-1}(b+c\alpha)^{-1} \end{pmatrix}
\end{align*}

We first note that the coset of $\mathbb{K}^{*0}$ need not be considered here, since all elements lie in the identity coset $C_0$. Then since $(\beta,\gamma) \vDash p_0|_{M,h_0,t_0,h}$, we see that $(\beta(b + c\alpha), \gamma(b+c\alpha) + c\beta^{-1})$ also realises $p_0|_{M,h_0,t_0,h}$ since $p_0 \in \mathcal{J}$.

We prove that $a + \frac{b^{-1}\alpha}{b + c\alpha}$ realises $p_{\infty,C_0}$. Write $\frac{b^{-1}\alpha}{b + c\alpha} = (b^2\alpha^{-1} + bc)^{-1}$. Then since $v(c) < dcl(M,b,\alpha) \cap \Gamma$, we see $v(b^2\alpha^{-1} + bc) < dcl(M,b,\alpha) \cap \Gamma$.

Hence $v((b^2\alpha^{-1} + bc)^{-1}) > dcl(M,b,\alpha) \cap \Gamma$. Hence $v(a + \frac{b^{-1}\alpha}{b + c\alpha}) = v(a) < \mathbb{Z}$, and hence $a + \frac{b^{-1}\alpha}{b + c\alpha} \vDash p_{\infty,C_0}$.

Hence $p_{\infty,C_0} * p_0$ is idempotent in $(S_G(M),*)$.
\end{proof}

Consider a type $q$ in $S_G(M)$. Then by using the group decomposition from Proposition \ref{Proposition_C((t))Decomp}, we see that we can express any realisation $g$ of $q$ in the form $g = zht$ for $z \in \mathbb{Z}/4\mathbb{Z}$, $h \in H(\bar{M})$ and $t \in B(\bar{M})$. The same can be done for any $g \in G(M)$, this time with $z \in \mathbb{Z}/4\mathbb{Z}$, $h \in H(M)$ and $t \in B(M)$.

{We now take the orbit-closure of $p_{\infty,C_0} * p_0$. Note that the orbit-closure of an idempotent element need not necessarily be minimal, though we do claim that $cl(G(M) * p_{\infty,C_0} * p_0)$ is indeed minimal, and will prove so later. We now compute the action of $G(M)$ on $p_{\infty,C_0} * p_0$, and do so by considering the action of $H(M)$, $B(M)$ and $\mathbb{Z}/4\mathbb{Z}$ separately.

\begin{Proposition} \label{Proposition_HMEquality}
 The $H(M)$-orbit of $p_{\infty,C_0} * p_0$ is $p_{\infty,C_0} * p_0$.
\end{Proposition}
\begin{proof}[\textbf{Proof}]
 Clearly, since $p_{\infty,C_0}$ is a minimal flow of the additive group, $H(M)$ acts trivially, and we see $H(M) * p_{\infty,C_0} * p_0 = p_{\infty,C_0} * p_0$.
\end{proof}

The following computations require us to be more precise about the cosets of $\mathbb{K}^{*0}$ than previously in the paper. In the following work, $p_0 \in \mathcal{J}$ will still mean the identity coset, and $p_i$, $p_j$... will denote arbitrary elements of $\mathcal{J}$. However, by $p_{k\mathbb{K}^{*0}}$, we mean a type in $\mathcal{J}$ realised by some $(\beta,\gamma)$ with $\beta, \gamma \in k\mathbb{K}^{*0}$.

Similarly, for one types, $C_0$ will still denote the identity coset, but when necessary we will be explicit and use $k\mathbb{K}^{*0}$ in place of $C_i$ to denote the specific coset of $\mathbb{K}^{*0}$ in which the realisations lie.

\begin{Proposition} \label{Proposition_BMLeftInclusion}
The $B(M)$-orbit of $p_{\infty,C_0} * p_0$ is a proper subset $V$ of $S_1(M) * \mathcal{J}$, where $V = \{ p_{\infty,k^2\mathbb{K}^{*0}} * p_{k\mathbb{K}^{*0}} \} \cup \{ p_{a,k^2\mathbb{K}^{*0}} * p_{k\mathbb{K}^{*0}} \text{ : } a \neq 0 \}$.
\end{Proposition}
 \begin{proof}[\textbf{Proof}]
Let $p_0 \in \mathcal{J}$. We compute $B(M) \cdot p_{\infty,C_0} * p_0$. 
Let $t_0 = \left( \begin{smallmatrix} b & c \\ 0 & b^{-1} \end{smallmatrix} \right)$ be an element of $B(M)$. \\
Let $h = \left( \begin{smallmatrix} 1 & 0 \\ \alpha & 1 \end{smallmatrix} \right)$ realise $p_{\infty,C_0}|_{M,t_0}$. \\
Let $t = \left( \begin{smallmatrix} \beta & \gamma \\ 0 & \beta^{-1} \end{smallmatrix} \right)$ realise $p_0|_{M,t_0,h}$.

Then $t_0 * p_{\infty,C_0} * p_0 = tp(t_0ht/M)$. We split into 2 cases; where $c = 0$ and $c \neq 0$. 

\textbf{Case 1:} Let $c = 0$.
  
 Then;
 \begin{align*}
   t_0ht & = \begin{pmatrix} b & c \\ 0 & b^{-1} \end{pmatrix} \begin{pmatrix} 1 & 0 \\ \alpha & 1 \end{pmatrix} \begin{pmatrix} \beta & \gamma \\ 0 & \beta^{-1} \end{pmatrix} \\
         & = \begin{pmatrix} b & 0 \\ 0 & b^{-1} \end{pmatrix} \begin{pmatrix} 1 & 0 \\ \alpha & 1 \end{pmatrix} \begin{pmatrix} \beta & \gamma \\ 0 & \beta^{-1} \end{pmatrix} \\
         & = \begin{pmatrix} 1 & 0 \\ b^{-2}\alpha & 1 \end{pmatrix} \begin{pmatrix} b \beta & b \gamma \\ 0 & \beta^{-1}b^{-1} \end{pmatrix}          
 \end{align*}

 Since $v(\alpha) < \mathbb{Z}$, and $v(b) \in \mathbb{Z}$, we see $v(b^{-2}\alpha) < \mathbb{Z}$. Further, since $\alpha \in \mathbb{K}^{*0}$, $b^{-2}\alpha \in b^{-2}\mathbb{K}^{*0}$. Hence $b^{-2}\alpha \vDash p_{\infty,b^{-2}\mathbb{K}^{*0}}$.

 Next, since $(\beta,\gamma) \vDash p_0$, which has $B(M)$-orbit $\mathcal{J}$ as $\mathcal{J}$ is minimal, we see that $(b \beta , b \gamma)$ realises $p_{b\mathbb{K}^{*0}} \in \mathcal{J}$.

 Hence, when $t_0 = (b,c) \in B(M)$ with $c = 0$, $t_0 * p_{\infty,C_0} * p_0 = p_{\infty,b^{-2}\mathbb{K}^{*0}} * p_k\mathbb{K}^{*0}$. 

 \textbf{Case 2:} Let $c \neq 0$. 

 Then;
 \begin{align*}
   t_0ht & = \begin{pmatrix} b & c \\ 0 & b^{-1} \end{pmatrix} \begin{pmatrix} 1 & 0 \\ \alpha & 1 \end{pmatrix} \begin{pmatrix} \beta & \gamma \\ 0 & \beta^{-1} \end{pmatrix} \\
         & = \begin{pmatrix} 1 & 0 \\ \frac{b^{-1}\alpha}{b + c\alpha} & 1 \end{pmatrix} \begin{pmatrix} \beta(b+c\alpha) & \gamma(b+c\alpha) + c\beta^{-1} \\ 0 & \beta^{-1}(b+c\alpha)^{-1} \end{pmatrix}          
 \end{align*}

 As $b\alpha^{-1} \neq 0$, we can write $\frac{b^{-1}\alpha}{b + c\alpha} = (b^2\alpha^{-1} + cb)^{-1} = (bc)^{-1}(1 + bc^{-1}\alpha^{-1})^{-1}$.

 Since $(1 + bc^{-1}\alpha^{-1})$ is in the infinitesimal neighbourhood of 1, which is itself a multiplicative group, we know that $(1 + bc^{-1}\alpha^{-1})^{-1}$ is of the form $1 + x$ where $v(x) \geq \mathbb{Z}$. Then $(1 + bc^{-1}\alpha^{-1})(1+x) = 1 + bc^{-1}\alpha^{-1} + x + xbc^{-1}\alpha^{-1}$.

 Since $(1 + bc^{-1}\alpha^{-1})(1+x) = 1$, we see $bc^{-1}\alpha^{-1} + x + xbc^{-1}\alpha^{-1} = 0$. Hence $x = -bc^{-1}\alpha^{-1} - xbc^{-1}\alpha^{-1}$, and since $v(x) \geq \mathbb{Z}$, and $b,c \in B(M)$, we see $v(x) = v(-bc^{-1}\alpha^{-1})$, and the coset of $\mathbb{K}^{*0}$ which contains $x$ is determined by $-bc^{-1}$.

 Hence  $\frac{b^{-1}\alpha}{b + c\alpha} = (bc)^{-1}(1 + x) = (bc)^{-1} + (bc^{-1})x$, and $x \in -bc^{-1}\mathbb{K}^{*0}$, and hence $((bc)^{-1})x \in c^{-2}\mathbb{K}^{*0}$.  

 Hence, $\frac{b^{-1}\alpha}{b + c\alpha} \vDash p_{(bc)^{-1}, c^{-2}\mathbb{K}^{*0}}$, where $(bc)^{-1} \neq 0 \in \mathbb{C}((t))$.

 Next, since $(\beta,\gamma) \vDash p_0$, which has $B(M)$-orbit $\mathcal{J}$ since $\mathcal{J}$ is a minimal flow of $(B,S_B(M))$, we see that $(\beta(b+c\alpha), \gamma(b+c\alpha) + c\beta^{-1})$ realises some $p_j \in \mathcal{J}$, where the coset $C_j$ of $\mathbb{K}^{*0}$ in which $\beta(b+c\alpha)$ and $\gamma(b+c\alpha) + c\beta^{-1}$ lie is determined by $(b+c\alpha)$. 

 Since $\alpha \vDash p_{\infty,C_0}$, we see that this coset is determined by $c$.

 Hence when $t_0 = (b,c)$ with $c \neq 0$, we see that $(b,c) \cdot p_{\infty,C_0} * p_0 = p_{(bc)^{-1}, c^{-2}\mathbb{K}^{*0}} * p_{c\mathbb{K}^{*0}}$.

 Note that since $b$ has no bearing on the cosets here, and $c \neq 0$, for any $a \in \mathbb{C}((t))$ and coset $c^{-2}\mathbb{K}^{*0}$ we can find an element $t_0 \in B(M)$ such that $t_0 \cdot p_{\infty,C_0} * p_0 = p_{a, c^{-2}\mathbb{K}^{*0}} * p_{c\mathbb{K}^{*0}}$. Namely, where $t_0 = ((ac)^{-1}, c)$.

 Hence, the $B(M)$-orbit of $p_{\infty,C_0} * p_0$ is a set of the form $\{ p_{\infty,k^{-2}\mathbb{K}^{*0}} * p_{k} \} \cup \{ p_{a,k^{-2}\mathbb{K}^{*0}} * p_{k} \text{ : } a \neq 0 \}$, where $p_k$ here means entries in the realisations of $p \in \mathcal{J}$ are elements of the coset $k\mathbb{K}^{*0}$.
\end{proof}

\begin{Proposition} \label{Proposition_BMEquality}
Let $V$ be as in the above proposition. Then $V$ is a subset of $B(M) \cdot p_{\infty,C_0} * p_0$ and hence $V = B(M) \cdot p_{\infty,C_0} * p_0$. 
\end{Proposition}
\begin{proof}[\textbf{Proof}]

We show that for any $q \in V$, there exists some $(b,c) \in B(M)$ such that $q = (b,c) * p_{\infty,C_0} * p_0$. 

We first consider the case where $q =  p_{\infty,k^{-2}\mathbb{K}^{*0}} * p_{k}$. Let $h \vDash p_{\infty,k^{-2}\mathbb{K}^{*0}}$ and let $t = (\beta,\gamma) \vDash p_{k}|{M,h}$. Then $ht \vDash q$.

Since $\beta, \gamma \in k\mathbb{K}^{*0}$, we can write $\beta = k\beta'$ and $\gamma = k\gamma'$, where $\beta'$ and $\gamma'$ lie in the coset $C_0 = \mathbb{K}^{*0}$. Since $k \in \mathbb{C}((t))$, we see $(\beta',\gamma') \vDash p_0|_{M,h}$.

Hence $(k,0) \cdot p_0 = p_k$.

\begin{align*}
 p_{\infty,k^2\mathbb{K}^{*0}} \cdot (k,0) & = \begin{pmatrix} 1 & 0 \\ \alpha & 1 \end{pmatrix} \begin{pmatrix} k & 0 \\ 0 & k^{-1} \end{pmatrix} \\
           & = \begin{pmatrix} k & 0 \\ k^{-1}\alpha & k^{-1} \end{pmatrix} \\
           & = \begin{pmatrix} k & 0 \\ 0 & k^{-1} \end{pmatrix} \begin{pmatrix} 1 & 0 \\ k^{2}\alpha & 1 \end{pmatrix}
\end{align*}

Finally, since $\alpha$ lies in $k^{-2}\mathbb{K}^{*0}$ we can write $\alpha = k^{-2}\alpha'$, for $\alpha' \in p_{\infty,C_0}$. Hence $k^{2}\alpha = k^{2}k^{-2}\alpha' = \alpha' \in \mathbb{K}^{*0}$.

Hence, for $q \in V$ of the form $p_{\infty,k^{-2}\mathbb{K}^{*0}} * p_{k}$, we can find an element $t_0$ of $B(M)$ such that $q = t_0 \cdot p_{\infty,C_0} * p_0$.


We now show that we can do the same when $q \in V$ is of the form $p_{a,k^{-2}\mathbb{K}^{*0}} * p_{k}$. Let $\alpha_0 \vDash p_{a,k^{-2}\mathbb{K}^{*0}}$, and let $(\beta,\gamma) \vDash p_k|_{M,\alpha_0}$. We show there exists some element $t_0 \in B(M)$ such that $q = t_0 * p_{\infty,C_0} * p_0$.

Since $\alpha_0 \vDash p_{a,k^{-2}\mathbb{K}^{*0}}$, we can write $\alpha_0 = a + \epsilon$, where $\epsilon \in k^{-2}\mathbb{K}^{*0}$ and $v(\epsilon) > \mathbb{Z}$. Further, we can write $\epsilon = -k^{-2}\epsilon'$ for some $\epsilon' \in \mathbb{K}^{*0}$ with $v(\epsilon') > \mathbb{Z}$.

Hence $\alpha_0 =  a(1 + a^{-1}\epsilon) = a(1 - a^{-1}k^{-2}\epsilon')$.

Let $b$ be such that $a = (bk)^{-1} \in \mathbb{C}((t))$.

Then $\alpha_0 = a(1 = bk^{-1}\epsilon')$. Now, $bk^{-1}\epsilon' \in -bk^{-1}\mathbb{K}^{*0}$, and $v( bk^{-1}\epsilon') > \mathbb{Z}$. Hence $ 1 - bk^{-1}\epsilon'$ is in the infinitesimal neighbourhood of $1$, and hence so is $(1 - bk^{-1}\epsilon')^{-1}$ since this neighbourhood forms a multiplicative group.

Then $(1 - bk^{-1}\epsilon')(1+x) = 1$, for some $x \in dcl(\mathbb{C}((t)), \alpha_0)$, and one can see that $x \in bk^{-1}\mathbb{K}^{*0}$ with $v(x) > \mathbb{Z}$.

Hence we can write $\alpha_0 = a(1 + bk^{-1}\alpha)^{-1}$ where $\alpha \in \mathbb{K}^{*0}$ and $v(\alpha) > \mathbb{Z}$. Hence $\alpha \vDash p_{0,C_0}$ and we can write $\alpha_0 = a(1 + bk^{-1}\alpha)^{-1} = \frac{a}{1 + bk^{-1}\alpha} = \frac{b^{-1}\alpha^{-1}}{b + k\alpha^{-1}}$, using $a = (bk)^{-1}$ from above.

Since $\mathcal{J}$ is a minimal flow of $(B(M),S_B(M))$, we can find some element $y \in B(\bar{M}) \cap dcl(M,\alpha_0)$ such that $y \cdot p_0 = p_k$.

We see that this $y = (b + k\alpha^{-1}, k)$ where $\alpha^{-1} \vDash p_{\infty,C_0}$, since $\alpha  \vDash p_{0,C_0}$ and $\alpha \in dcl(M,\alpha_0)$. 

Further, we see that the cosets of $\mathbb{K}^{*0}$ in which the entries lie remain unchanged, since $\alpha^{-1} \in \mathbb{K}^{*0}$ and $v(\alpha^{-1}) < \mathbb{Z}$, we have $b + k\alpha^{-1} \in k\mathbb{K}^{*0}$. 

\begin{align*}
 \begin{pmatrix} 1 & 0 \\ \alpha_0 & 1 \end{pmatrix} \begin{pmatrix} \beta & \gamma \\ 0 & \beta^{-1} \end{pmatrix} & = \begin{pmatrix} 1 & 0 \\ \alpha_0 & 1 \end{pmatrix} \begin{pmatrix} b + k_\alpha^{-1} & k \\ 0 & (b + k\alpha^{-1})^{-1} \end{pmatrix} \begin{pmatrix} \beta' & \gamma ' \\ 0 & \beta'^{-1} \end{pmatrix} \\
 & = \begin{pmatrix} b + k\alpha^{-1} & k \\ \alpha_0(b + k\alpha^{-1}) & \alpha_0 k + (b + k\alpha^{-1})^{-1}  \end{pmatrix} \begin{pmatrix} \beta' & \gamma ' \\ 0 & \beta'^{-1} \end{pmatrix} \\
& = \begin{pmatrix} b + k\alpha^{-1} & k{-1} \\ \frac{b^{-1}\alpha^{-1}}{b + k\alpha^{-1}}(b + k\alpha^{-1}) & \frac{b^{-1}\alpha^{-1}}{b + k\alpha^{-1}}k + (b + k\alpha^{-1})^{-1} \end{pmatrix} \begin{pmatrix} \beta' & \gamma ' \\ 0 & \beta'^{-1} \end{pmatrix} \\
& = \begin{pmatrix} b + k\alpha^{-1} & k \\ b^{-1}\alpha^{-1} & \frac{b^{-1}\alpha^{-1}}{b + k\alpha^{-1}}k + (b + k\alpha^{-1})^{-1} \end{pmatrix} \begin{pmatrix} \beta' & \gamma ' \\ 0 & \beta'^{-1} \end{pmatrix} \\
 & = \begin{pmatrix} b + k\alpha^{-1} & k \\ b^{-1}\alpha^{-1} & b^{-1} \end{pmatrix} \begin{pmatrix} \beta' & \gamma ' \\ 0 & \beta'^{-1} \end{pmatrix} \\
& = \begin{pmatrix} b & k \\ 0 & b^{-1} \end{pmatrix} \begin{pmatrix} 1 & 0 \\ \alpha^{-1} & 1 \end{pmatrix} \begin{pmatrix} \beta' & \gamma' \\ 0 & \beta^{-1} \end{pmatrix}.
\end{align*}

Where $(\beta',\gamma')$ realise $p_0$ over $(M,\alpha_0)$. 

Then $(b,k) = t_0 \in B(M)$, $\alpha^{-1} \vDash p_{\infty,C_0}$. Further, $(\beta',\gamma') \vDash p_0|_{M,\alpha_0}$ and since $\alpha^{-1} \in dcl(M,\alpha_0)$, we have $(\beta',\gamma') \vDash p_0|_{M,\alpha^{-1}}$ as required.

Hence for any $q \in V$ of the form $p_{a,k^{-2}\mathbb{K}^{*0}} * p_{k}$, we can find some $t_0 \in B(M)$ such that $q = t_0 \cdot p_{\infty,C_0} * p_0$.

Hence $V \subseteq B(M) \cdot p_{\infty,C_0} * p_0$ and by \ref{Proposition_BMLeftInclusion}, we see $V = B(M) \cdot p_{\infty,C_0} * p_0$.
\end{proof}

\begin{Proposition} \label{Proposition_Z4ZLeftInclusion}
 The union $\bigcup\limits_{v \in V} \mathbb{Z}/4\mathbb{Z} \cdot v$ is a subset of $V \cup \{ (p_{0,k^{-2}\mathbb{K}^{*0}} * p_{k^3\mathbb{K}^{*0}}) \text{ : } k \in \mathbb{C}((t))^{*} \}$

\end{Proposition}
 \begin{proof}[\textbf{Proof}]
 Let $z \in \mathbb{Z}$.  Clearly if $z = I_2$, the identity of $\mathbb{Z}/4\mathbb{Z}$, then $z \cdot v = v$ for all $v \in V$.

 We split into 2 cases. First let $h \vDash p_{\infty,k^2\mathbb{K}^{*0}}$ and $t \vDash p_{k\mathbb{K}^{*0}}|_{M,h}$ for some $k \in \mathbb{C}((t))$. Then $ht \vDash v$ for some $v \in V$.

 Suppose $z = -I_2$. Since $-I_2$ is in the centre of $SL_2$, we can write $z \cdot  p_{\infty,k^2\mathbb{K}^{*0}} * p_{k\mathbb{K}^{*0}} = p_{\infty,k^2\mathbb{K}^{*0}} * z \cdot p_{k\mathbb{K}^{*0}}$.

 Then $\mathcal{J}$ is a minimal subflow of $(B(M),S_B(M))$ and so $z \cdot p_{k\mathbb{K}^{*0}} = p_j$ for some $p_j \in \mathcal{J}$. However, since $-1 \in \mathbb{K}^{*0}$, every entry of $z \cdot p_{k\mathbb{K}^{*0}}$ is in the same coset as the corresponding entry in $p_{k\mathbb{K}^{*0}}$, and so $z \cdot p_{k\mathbb{K}^{*0}} = p_{k\mathbb{K}^{*0}}$ and hence $z \cdot p_{\infty,k^2\mathbb{K}^{*0}} * p_{k\mathbb{K}^{*0}} = p_{\infty,k^2\mathbb{K}^{*0}} * p_{k\mathbb{K}^{*0}}$.

 Finally, suppose $z = \begin{pmatrix} 0 & -1 \\ 1 & 0 \end{pmatrix}$.

 Then $z \cdot p_{\infty,k^2\mathbb{K}^{*0}} * p_{k\mathbb{K}^{*0}} = tp(zht/M)$ and we see;
  \begin{align*}
   zht & = \begin{pmatrix} 0 & -1 \\ 1 & 0 \end{pmatrix} \begin{pmatrix} 1 & 0 \\ \alpha & 1 \end{pmatrix} \begin{pmatrix} \beta & \gamma \\ 0 & \beta^{-1} \end{pmatrix} \\
       & = \begin{pmatrix} -\alpha & -1 \\ 1 & 0 \end{pmatrix} \begin{pmatrix} \beta & \gamma \\ 0 & \beta^{-1} \end{pmatrix} \\
       & = \begin{pmatrix} 1 & 0 \\ -\alpha^{-1} & 1 \end{pmatrix} \begin{pmatrix} -\alpha & -1 \\ 0 & -\alpha^{-1} \end{pmatrix} \begin{pmatrix} \beta & \gamma \\ 0 & \beta^{-1} \end{pmatrix} \\
       & = \begin{pmatrix} 1 & 0 \\ -\alpha^{-1} & 1 \end{pmatrix} \begin{pmatrix} -\alpha\beta & -\alpha\gamma - \beta^{-1} \\ 0 & -\alpha^{-1}\beta^{-1} \end{pmatrix}
  \end{align*}

 Observe that $-\alpha^{-1} \in k^{-2}\mathbb{K}^{*0}$ and $v(-\alpha^{-1}) > \mathbb{Z}$. Further, $-\alpha\beta \in k^3\mathbb{K}^{*0}$, $-\alpha\gamma \in k^3\mathbb{K}^{*0}$.

 Hence $-\alpha^{-1} \vDash p_{0,k^{-2}\mathbb{K}^{*0}}$ and $(-\alpha\beta, -\alpha\gamma-\beta^{-1}) \vDash p_{k^3\mathbb{K}^{*0}}|_{M,-\alpha^{-1}}$.

 The case where $z = \begin{pmatrix} 0 & 1 \\ -1 & 0 \end{pmatrix}$ follows similarly. Taking the union over elements of $V$ we see that $\bigcup\limits_{v \in V} \mathbb{Z}/4\mathbb{Z} \cdot v \subseteq V \cup \{ (p_{0,k^{-2}\mathbb{K}^{*0}} * p_{k^3\mathbb{K}^{*0}}) \text{ : } k \in \mathbb{C}((t))^{*} \}$.
 \end{proof}
 
 \begin{Proposition} \label{Proposition_Z4ZEquality}
 $V \cup \{ (p_{0,k^{-2}\mathbb{K}^{*0}} * p_{k^3\mathbb{K}^{*0}}) \text{ : } k \in \mathbb{C}((t))^{*} \}$ is a subset of $\bigcup\limits_{v \in V} \mathbb{Z}/4\mathbb{Z} \cdot v$.
 \end{Proposition}
 \begin{proof}[\textbf{Proof}]

Clearly any element $v \in V$ lies in the set $\bigcup\limits_{v \in V} \mathbb{Z}/4\mathbb{Z} \cdot v$ since we could choose $z \in \mathbb{Z}/4\mathbb{Z}$ to be the identity element.

Hence we just need to show that $(p_{0,k^{-2}\mathbb{K}^{*0}} * p_{k^3\mathbb{K}^{*0}})$ can be expressed in the form $z \cdot v$ for some $z \in \mathbb{Z}/4\mathbb{Z}$ and some $v \in V$.

Fix some arbitrary non-zero $k \in \mathbb{C}((t))$. Let $h = \begin{pmatrix} 1 & 0 \\ \alpha & 1 \end{pmatrix}$, where $\alpha \vDash p_{0,k^{-2}\mathbb{K}^{*0}}$. Let $t = (\beta,\gamma) \vDash p_{k^{3}\mathbb{K}^{*0}}|_{M,h}$.

Since $\mathcal{J}$ is a minimal subflow of $(B(M),S_B(M))$, and $t \vDash p_{k^{3}\mathbb{K}^{*0}}|_{M,h}$, we can find some $(b,c) \in B(\bar{M}) \cap dcl(M,h)$ such that $(b,c) * p_{k\mathbb{K}^{*0}}|_{M,h} = p_{k^{3}\mathbb{K}^{*0}}$.

In particular, we know that $b$ uniquely determines the coset in this factorisation. Note that since $\alpha \vDash p_{0,k^{-2}\mathbb{K}^{*0}}$, we see $-\alpha^{-1} \vDash p_{\infty,k^2\mathbb{K}^{*0}}$. As such, we can choose $(b,c) = (-\alpha^{-1}, -1)$, and hence;

 \begin{align*}
   ht & = \begin{pmatrix} 1 & 0 \\ \alpha & 1 \end{pmatrix} \begin{pmatrix} \beta & \gamma \\ 0 & \beta^{-1} \end{pmatrix} \\
      & = \begin{pmatrix} 1 & 0 \\ \alpha & 1 \end{pmatrix} \begin{pmatrix} -\alpha^{-1} & -1 \\ 0 & -\alpha \end{pmatrix} \begin{pmatrix} \beta' & \gamma' \\ 0 & \beta'^{-1} \end{pmatrix} \\
      & = \begin{pmatrix} -\alpha^{-1} & -1 \\ 1 & 0 \end{pmatrix} \begin{pmatrix} \beta' & \gamma' \\ 0 & \beta'^{-1} \end{pmatrix} \\
      & = \begin{pmatrix} 0 & -1 \\ 1 & 0 \end{pmatrix} \begin{pmatrix} 1 & 0 \\ \alpha^{-1} & 1 \end{pmatrix} \begin{pmatrix} \beta' & \gamma' \\ 0 & \beta'^{-1} \end{pmatrix} \\
      & = zh't'
  \end{align*}

 Where $z \in \mathbb{Z}/4\mathbb{Z}$, $h' \vDash p_{\infty,k^2\mathbb{K}^{*0}}$ and $t' \vDash p_{k\mathbb{K}^{*0}}|_{M,h}$.

 Since $h' \in dcl(M,h)$, we also see $t' \vDash p_{k\mathbb{K}^{*0}|_{M,h'}}$, and so $zh't' \vDash z \cdot p_{\infty,k^2\mathbb{K}^{*0}} * p_{k\mathbb{K}^{*0}}$, which is an element of $\bigcup\limits_{v \in V} \mathbb{Z}/4\mathbb{Z} \cdot v$ as required.
\end{proof}

Hence, we have demonstrated that the $G(M)$-orbit of $p_{\infty,C_0} * p_0$ is precisely $V' = V \cup \{ (p_{0,k^{-2}\mathbb{K}^{*0}} * p_{k^3\mathbb{K}^{*0}}) \text{ : } k \in \mathbb{C}((t))^{*0} \}$. This set is not closed, and hence not a subflow minimal or otherwise. However, we can take the closure of $V'$ and we claim that this is indeed minimal, as we now show.

Taking closures of both sides, we see that $cl(G(M) \cdot p_{\infty,C_0} * p_0) = cl(V')$. By Fact \ref{Fact_Gorbit}, for any type $p$, $cl(G(M) * p) = S_G(M) * p$. Hence the closure of the $G(M)$-orbit of $p_{\infty,C_0} * p_0$ is the $S_G(M)$-orbit of $p_{\infty,C_0} * p_0$. Hence $S_G(M) \cdot p_{\infty,C_0} * p_0 = cl(V')$.

\begin{Lemma}
  $S_G(M) * p_{\infty,C_0} * p_0 \subseteq S_1(M) * \mathcal{J}$.

Moreover, every element $s * p_{\infty,C_0} * p_0$ is of the form $r * p$ with $r \in S_1(M)$, $p \in \mathcal{J}$.
\end{Lemma}
\begin{proof}[\textbf{Proof}]
We can see from Propositions \ref{Proposition_HMEquality}, \ref{Proposition_BMEquality} and \ref{Proposition_Z4ZEquality} and see that the $G(M)$-orbit of $p_{\infty,C_0} * p_0 = V' \subset S_1(M) * \mathcal{J}$. The $S_G(M)$-orbit behaves in a similar way, though it will be a proper superset of $V'$, however still a subset of $S_1(M) * \mathcal{J}$. 
 \end{proof}
 
\begin{Lemma}
Let $p, p' \in \mathcal{J}$. Let $q \in S_1(M)$. Then $p * q * p'$ can be expressed as an element of $S_1(M) / \{ p_{\infty,C_k} \text{ : } k \in \mathbb{C}((t)) \} * \mathcal{J}$.
\end{Lemma}
\begin{proof}[\textbf{Proof}]
  This is easy to see using the same method as in Proposition \ref{Proposition_BMEquality}, and note that insisting $p \in \mathcal{J}$ removes the case where $c = 0$. 
\end{proof}

 \begin{Theorem}
  $cl(V')$ is a minimal subflow of $(G(M),S_G(M))$.
 \end{Theorem}
 \begin{proof}[\textbf{Proof}]
  Any point in $cl(V')$ is of the form $s * p_{\infty,C_0} * p_0$, and by the above lemma we can show any type of the form $s * p_{\infty,C_0} * p_0$ can be expressed as an element $q * p$ of $S_1(M) * \mathcal{J}$.

  We claim that for any $r$ in $S_1(M) * \mathcal{J}$, we can demonstrate that $p_{\infty,C_0} * p_0$ is in the orbit-closure $cl(G(M)*r) = S_G(M) * r$.

  Let $r = q * p \in S_1(M) * \mathcal{J}$. 

  Then we can find some type $p_{\infty,C_0} \cdot p_j$, such that $p_{\infty,C_0} * p_j * q * p = p_{\infty,C_0} * q' * p_0$, using a similar argument to Proposition \ref{Proposition_BMEquality}. However, here we see that the realisation $(b,c)$ of the heir of $p_j$ ensures that $q' \notin \{ p_{\infty,C_k} \text{ : } k \in \mathbb{C}((t))^* \}$. As such, $p_{\infty,C_0} * q' = p_{\infty,C_0}$.
  
  Hence we can find some $s \in S_G(M)$ such that $s * r = p_{\infty,C_0} * p_0$, and so $p_{\infty,C_0} * p_0$ is in the orbit-closure of $r$ for any $r \in cl(V')$.
  
  Since $p_{\infty,C_0} * p_0$ is in the $S_G(M)$-orbit of any element of $cl(V')$, and $cl(V') = S_G(M) * p_{\infty,C_0} * p_{\infty}$, we see that the $cl(V')$ is the orbit-closure of any type in $cl(V')$, and hence minimal.
 \end{proof}

\section{The Ellis Group of $SL_2(\mathbb{C}((t)))$}

To obtain the Ellis Group of $(G(M),S_G(M))$, from \ref{Theorem_Ellis}, we act on the minimal subflow of $(G(M),S_G(M))$ (equivalently, the minimal closed left ideal of $(S_G(M), *)$ ) by an idempotent, namely $p_{\infty,C_0} * p_0$. 

\begin{Theorem} \label{Theorem_EllisGroupC((t))}
 The Ellis Group of $(G(M),S_G(M))$ is $p_{\infty,C_0} * p_0 * cl(V') = p_{\infty,C_0} * \mathcal{J}$, and is isomorphic to $B/B^0$.
\end{Theorem}
\begin{proof}[\textbf{Proof}]
    This is clear to see. Take any element $r = q * p \in cl(V')$. We compute $p_{\infty,C_0} * p_0 * r$.

   We note that $p_0 * r$ is of the form $q' * p'$ for some $q' \in S_1(M)/ \{ p_{\infty,C_k} \text{ : } k \in \mathbb{C}((t))^* \}$ and $p' \in \mathcal{J}$.

   Then $p_{\infty,C_0} * p_0 * q * p = p_{\infty,C_0} * p'$ for some $p' \in \mathcal{J}$, and hence $p_{\infty,C_0} * p_0 * cl(V') \subseteq p_{\infty,C_0} * \mathcal{J}$.

  To demonstrate equality, we must show that $p'$ can range over all cosets of $\mathbb{K}^{*0}$. That is, for any $p_{\infty,C_0} * p_j$, we can find some $r \in cl(V')$ with $p_{\infty,C_0} * p_0 * r = p_{\infty,C_0} * p_j$.

  This is clear to see from Proposition \ref{Proposition_BMEquality}. Since $V \subset cl(V')$, we see that types of the form $p_{\infty,k^2\mathbb{K}^{*0}} * p_k$ for all $k \in \mathbb{C}((t))$ are contained $cl(V')$. One can simply choose $r = p_{\infty,j^2\mathbb{K}^{*0}} * p_j$ and show $p_{\infty,C_0} * p_0 * r = p_{\infty,C_0} * p_j$. Hence $p_{\infty,C_0} * \mathcal{J} \subseteq p_{\infty,C_0} * p_0 * cl(V')$.

 Hence $p_{\infty,C_0} * p_0 * cl(V') = p_{\infty,C_0} * \mathcal{J}$, and so the Ellis Group of $(G(M),S_G(M))$ is precisely $p_{\infty,C_0} * \mathcal{J}$. Since $\mathcal{J}$ is isomorphic to $B/B^0$, it is easy to see that $p_{\infty,C_0} * \mathcal{J}$ is isomorphic to $B/B^0$ also.
\end{proof}

Hence, we have demonstrated that the Ellis Group of $(G(M),S_G(M))$ is not isomorphic to $G/G^{00}$; namely the Ellis Group has infinite elements whereas $G/G^{00}$ is trivial.

We discuss briefly how we believe this may generalise to a larger class of metastable definable groups. We noted that $SL_2(\mathbb{C}((t)))$ admits an Iwasawa-like decomposition, though as $SL_2(\mathbb{C}[[t]])$ was not even locally compact we could not adopt a similar approach to that of \textbf{\cite{SL2QP}}. However, using stable domination one can show that $SL_2(\mathbb{C}[[t]])$ admits a unique 2-sided global generic, $q$, whose restriction to $\mathbb{C}$ is generic in $SL_2(\mathbb{C})$ (This fact is originally shown for $ACVF$ in \textbf{\cite{StabDom}}).

Using the results of \textbf{\cite{HR2017}} one can see that stably dominated groups are definably amenable (indeed in \textbf{\cite{HPP}} there is mention that these groups are in fact $fsg$ groups). This provides access to a generalisation in this setting if one is capable of dealing with the infinite intersection in the decomposition $SL_n(K) = B \times SL_n(\mathcal{O}_K)$. In this setting, since we lose information about the cosets of $P_n$ predicates due to algebraic closure, we may expect that $SL_2(K)$ would admit a trivial Ellis Group and hence be isomorphic to $SL_2(K)/SL_2(K)^{00}$. We suggest that a general description could be found for metastable definable groups that admit a maximally stably dominated - definably amenable group decomposition.

\section{Acknowledgements}

The author would like to thank his supervisors Dr. Davide Penazzi, Dr. Charlotte Kestner and Dr. Marcus Tressl for their assistance and support in this work. The author would also like to thank Dr. Sylvy Anscombe for many helpful discussions. Finally, the author would like to extend thanks to the reviewer for exceptionally thorough and helpful feedback.

\bibliographystyle{plain}
\bibliography{References}

\end{document}